\documentclass[12pt]{amsart}
\usepackage{amssymb,latexsym,amsmath,amscd,mathrsfs,yfonts,hyperref}
\usepackage{fullpage}
\input xy
\xyoption{all}

\newcommand{\ZZ}{\mathbb Z}

\newcommand{\CC}{\mathbb C}

\newcommand{\cpt}{\mathbb K}

\def\bK{\mathbf K}

\def\Aut{\textup{Aut}}

\def\cos{\textup{cos}}
\def\sin{\textup{sin}}
\def\Coh{\mathtt{Coh}}
\def\RHom{\textup{RHom}}

\def\GL{\textup{GL}}
\def\SL{\textup{SL}}
\def\Int{\textup{Int}}
\def\per{\textup{per}}

\def\Pro{\textup{Pro}}

\def\dg{\textup{dg}}
\def\Cone{\textup{Cone}}
\def\C{\textup{C}}

\def\KKcat{\mathtt{KK_{C^*}}}
\def\Z{\textup{Z}}

\def\ev{\textup{ev}}

\def\id{\mathrm{id}}

\def\End{\textup{End}}

\def\Hom{\textup{Hom}}

\def\PHoSS{\mathtt{HoSSet_{prop}}}
\def\SSets{\mathtt{SSet}}
\def\SSpro{\mathtt{SSet_{prop}}}
\def\ker{\textup{ker}}

\def\op{\textup{op}}

\def\sd{\textup{Sd}}
\def\rep{\textup{rep}}

\def\Ab{\sf Ab}

\def\Hmo{\mathtt{Hmo}}

\def\HoPA{\mathtt{{HoProC^*}^{stab}_{prop}}}

\def\K{\textup{K}}

\def\NCS{\mathtt{NCS_{dg}}}
\def\DGcorr{\mathtt{NCC^{{\mathbf K}}_{dg}}}

\def\CSp{\mathtt{NCS_{C^*}}}

\def\NCC{\mathtt{NCC}}
\def\HoSpt{\mathtt{HoSpt}}

\def\DGcat{\mathtt{DGcat}}

\def\PAlg{\mathtt{ProC^*_{prop}}}
\def\SPAlg{\mathtt{{ProC^*}^{stab}_{prop}}}
\def\KK{\textup{KK}}
\def\E{\textup{E}}

\def\H{\textup{H}}
\def\End{\textup{End}}

\def\Sets{\mathtt{Set}}

\def\cC{\mathcal C}

\def\1{\bf{1}}

\def\Csep{\mathtt{Sep_{C^*}}}

\def\Pr{{\sf Top_{dg}^{{\bK}}}}

\def\cRep{{\sf{Rep_{proC^*}}}}
\def\ND{\textup{Nd}}

\def\cD{\mathcal D}

\def\DGR{{\sf{Top_{fib}^{{\bK}}}}}
\def\cM{\mathcal M}

\def\TT{\mathbb{T}}

\def\A{\mathcal{A}}
\def\cA{\mathcal{A}}

\newcommand{\map}{\rightarrow}
\newcommand{\functor}{\longrightarrow}

\newcommand{\beq}{\begin{eqnarray}}
\newcommand{\beqn}{\begin{eqnarray*}}
\newcommand{\eeq}{\end{eqnarray}}
\newcommand{\eeqn}{\end{eqnarray*}}

\newtheorem{thm}{Theorem}[section]

\newtheorem{lem}[thm]{Lemma}
\newtheorem{prop}[thm]{Proposition}
\newtheorem{cor}[thm]{Corollary}
\newtheorem{ex}[thm]{Example}
\newtheorem{defn}[thm]{Definition}
\newtheorem{rem}[thm]{Remark}

\begin{document}

\title{Noncommutative correspondence categories, simplicial sets and pro $C^*$-algebras}
\author{Snigdhayan Mahanta}

\maketitle


\begin{abstract}
We construct an additive functor from the category of separable $C^*$-algebras with morphisms enriched over Kasparov's $\KK_0$-groups to the noncommutative correspondence category $\DGcorr$, whose objects are small DG categories and morphisms are given by the equivalence classes of some DG bimodules up to a certain $\K$-theoretic identification. Motivated by a construction of Cuntz we associate a pro $C^*$-algebra to any simplicial set, which is functorial with respect to proper maps of simplicial sets and those of pro $C^*$-algebras. This construction respects homotopy between proper maps after enforcing matrix stability on the category of pro $C^*$-algebras. The first result can be used to deduce derived Morita equivalence between DG categories of topological bundles associated to separable $C^*$-algebras up to a $\K$-theoretic identification from the knowledge of $\KK$-equivalence between the $C^*$-algebras. The second construction gives an indication that one can possibly develop a {\it noncommutative proper homotopy theory} in the context of topological algebras, e.g., pro $C^*$-algebras. 

\end{abstract}

\begin{center}
{\bf Introduction}
\end{center}

Experts believe that in any category of noncommutative spaces correspondence-like morphisms should also be included, e.g., `motivated morphisms' in \cite{MMM}. Such morphisms are given by some bimodules or generalizations thereof that induce well-defined morphisms on most (co)homological constructions that we know. Such correspondences can be naturally seen as morphisms over some desirable enrichments of the category of noncommutative spaces, for instance, additive or spectral ones. Furthermore, most of the interesting (co)homological invariants that we know should factor through this category. The philosophy closely adheres to that of motivic homotopy theory \cite{FSV} but including noncommutative objects. More precisely as a category of noncommutative correspondences in the operator algebraic setting we consider the well-known category $\KKcat$ with Kasparov's bivariant $\K$-theory as morphisms, which, in fact, has a triangulated structure \cite{MeyNes}. This category appears in the context of noncommutative motives in \cite{ConConMar1}. The additivized Morita homotopy category of DG (differential graded) categories $\Hmo_0$, whose construction was outlined in \cite{TabThesis}, is potentially a good candidate for the category of noncommutative correspondences in the setting of DG categories. Kontsevich calls the same category as the category of {\it noncommutative motives} in \cite{KonMot}. All {\it additive invariants} of DG categories factor through this category \cite{TabThesis}. However, for the applications that we have in mind a more decisive localization is needed. In certain applications to (co)homological duality statements in physics one would like to construct morphisms which induce isomorphisms between topological $\K$-theories (or other generalized cohomology theories). Typically one is interested in topological $\K$-theory since D-brane charges are classified by them. Therefore, we invert morphisms which induce homotopy equivalences of Waldhausen's $\K$-theory spectra \cite{Tab3} of DG categories and further perform a group completion of the morphisms. Although we have chosen to localize morphisms which induce homotopy equivalences of $\K$-theory spectra, we refer to the resulting category henceforth simply as {\it noncommutative correspondence category} suppressing the dependence on the choice of the localization in order to avoid verbiage. Nevertheless we make the dependence explicit in the notation as it is denoted by $\DGcorr$. It is plausible that for different applications localizing along different sets of morphisms can be useful. From our point of view, one can obtain duality results fairly easily if one is less ambitious, i.e., instead of trying to prove equivalence of DG or $\textup{A}_\infty$-categories on the nose if one tries to prove them after suitable localization one can be more successful.

In the first section we begin with the description of the noncommutative DG correspondence category $\DGcorr$ and the noncommutative $C^*$-correspondence category $\KKcat$. Then we construct our additive functor $\DGR:\KKcat\functor\DGcorr$ (see Theorem \ref{MainThm}). The functor $\DGR(\pounds)$ roughly sends a unital $C^*$-algebra to its bounded DG category of complexes of vector bundles (or finitely generated projective modules). At the heart of this construction lies Quillen's description of the $K_0$-group of nonunital algebras \cite{QuiNonunitalK0}, which is particularly well behaved for $C^*$-algebras. The interest lies mostly in isomorphisms and by functoriality this result implies that an isomorphism in $\KKcat$ would translate to an isomorphism of DG categories in $\DGcorr$. It should be clarified that isomorphisms in $\DGcorr$ are weaker than those in $\Hmo_0$. In the classification programme of $C^*$-algebras the Kirchberg--Phillips Theorem states that two stable {\it Kirchberg algebras} are $\KK$-equivalent if and only if they are $*$-isomorphic \cite{KirPhi,NCPClassification} and therefore $\KK$-equivalent stable Kirchberg algebras will have isomorphic (via DG functors) DG categories of topological bundles.
This result is probably by itself not very interesting as it is purely a topological statement. However, by incorporating more structures into a $C^*$-algebra (for example, a curved topological DGA) one can hopefully produce more instances of noncommutative dualities as in \cite{Blo1,Blo2,BloDan} simply from the well studied $\KK$-isomorphisms. Purely in the operator algebraic context the connection between $\KK$-dualities and noncommutative $\textup{T}$-dualities have been explored in \cite{BMRS,BMRS2}. The functor $\DGR$, although not full, seems to be related to some nontrivial isomorphisms like the Fourier--Mukai type dualities for tori (see Example \ref{TorusEx}). One can also deduce that for any nuclear separable $C^*$-algebra $\pounds$ the DG category $\DGR(\pounds)$ is invariant in $\DGcorr$ under {\it strong deformation} if $\pounds$ and its strong deformation are suitably homotopy equivalent (see subsection \ref{deformation}). Similar results at the level of $\K$-theory groups already exist in the literature \cite{DadLor} (see also \cite{RosKThQuant} for a survey). We also discuss the concept of homological $\TT$-dualities and its connection with that of topological $C^*$-correspondences briefly (see subsection \ref{T-duality}). It would be interesting to extend our construction of $\DGR(\pounds)$ to the context of purely algebraic bivariant $\K$-theory developed by Corti{\~n}as--Thom \cite{CorTho}, at least in the $\H$-unital case. 

The second section is inspired by some emergent connections between homotopy theory and {\it noncommutative topology}. Several problems in noncommutative topology, including the Baum--Connes conjecture, have some level of properness built into them. From our point of view the objects themselves need not be finite (or compact) but the maps interconnecting them should be proper in some suitable manner and compose well to form a category. Cuntz constructed a universal noncommutative $C^*$-algebra using generators and relations from a locally finite simplicial complex to give a conceptual understanding of the Baum--Connes assembly map in \cite{CunSimp}. We associate a noncommutative pro $C^*$-algebra to any simplicial set (without any finiteness assumption), which is functorial with respect to proper maps of simplicial sets (see Definition \ref{PropMap}). Then we show that our construction induces a functor between the category of simplicial sets with, what we call, proper homotopy classes of proper maps between them and the matrix stabilized category of pro $C^*$-algebras with homotopy classes of proper or nondegenerate maps between them. The constructions in this section are admittedly not always entirely satisfactory. Although we mostly deal with the category of pro $C^*$-algebras with proper $*$-homomorphisms between them, in order to enforce matrix stability we enlarge the morphisms by corner embeddings, which are not proper, and invert these maps formally. However, the formal inverses of the corner embeddings are in some sense proper. Another sticking point is that we cannot ensure that our construction produces continuous $*$-homomorphisms between pro $C^*$-algebras. Therefore we work with arbitrary $*$-homomorphisms which become automatically continuous if the domain pro $C^*$-algebra is actually a $\sigma$-$C^*$-algebra or a countable inverse limit of $C^*$-algebras. It seems plausible that without enlarging the category of $C^*$-algebras to include certain limits of topological $*$-algebras, it is not possible to construct a Quillen model category structure on it. After appropriate enlargement (containing pro $C^*$-algebras) there does exist a cofibrantly generated model category structure with $\KK_*$-equivalences (or $\K_*$-equivalences) as weak equivalences \cite{JoaJoh}. Our result might suggest that the matrix stabilized category of pro $C^*$-algebras with proper maps between them is amenable to {\it noncommutative proper (or infinite) homotopy theory} as explained in, e.g., \cite{BauQui}. The author recently learnt that a very general framework of homotopy theory in the context of $C^*$-algebras has been developed by {\O}stv{\ae}r \cite{Ost}.

Although the two sections presented in this article seem unrelated the author hopes to make the connection clearer in future.

A general remark is in order here. A $C^*$-algebra or an abstract DG category is not very geometric in nature. The appropriate objects for geometry should be something close to Connes' spectral triples (see, e.g., the reconstruction Theorem \cite{STRec}) and presumably a {\it pretriangulated} DG category whose homotopy category is {\it geometric} in the sense of Kontsevich \cite{KonNotes}. Therefore, our results should be viewed in the realm of noncommutative topology on which interesting geometric structures can be built. In particular, our correspondence categories should be regarded as {\it topological} correspondence categories. In order to make the constructions a bit more sensitive to the norm structures on the topological algebras, one might consider replacing the DG category of $\CC$-linear spaces (noncommutative point) throughout by the DG category associated to the category of locally convex topological vector spaces, which admits the structure of a {\it quasiabelian} category \cite{Schn,Pros}. 

\vspace{3mm}
\noindent
{\bf Notations and conventions:} We do not assume our algebras to be commutative or unital unless explicitly stated so. In Section \ref{NCC} we require our $C^*$-algebras to be separable, which from the point of view of topology requires the spaces to be {\it metrizable}. Many constructions in geometry require paracompactness (or an argument involving partition of unity) and from that perspective the separability assumption is quite natural. Moreover, the technical issues of $\KK$-theory are properly understood only in the separable case. The focus of Section \ref{Homotopy} is combinatorial topology and hence we do away with the separability assumption and in fact we work with a larger category of pro $C^*$-algebras or inverse limit $C^*$-algebras.  We are going to work over a ground field $k$ and while working with $C^*$-algebras the field $k$ will be tacitly assumed to be $\CC$. Unless otherwise stated, all functors are also assumed to be covariant and appropriately derived (whenever necessary), although we shall use the underived notation for brevity, e.g., we shall write $\otimes$ for $\otimes^{\textup{L}}$. The tensor products of $C^*$-algebras are suitably completed and since in all cases one of the algebras is nuclear we do not need to worry about the distinction between maximal and minimal tensor products. Throughout this article we make use of the language of model categories and simplicial homotopy theory whose details we have left out. The standard references for them are \cite{HomotopicalAlgebra,Hov,Hir,GoeJar}.

\vspace{3mm}
\noindent
{\bf Acknowledgements.} The author is extremely grateful to P. Goerss, B. Keller, Matilde Marcolli, R. Meyer, Fernando Muro, J. Rosenberg, G. Tabuada and B. To{\"e}n for several email correspondences answering various questions. The author would like to thank Beno{\^i}t Jacob, S. Krishnan and M. Schlichting for some helpful discussions. The author is particularly indebted to A. Connes, G. Elliott and R. Meyer for pointing out several inaccuracies in an initial draft of the article. The mistakes that might remain are solely the author's responsibility. The author also gratefully acknowledges the hospitality of the department of mathematics at the University of Toronto, Fields Institute and Institut des Hautes \'Etudes Scientifiques, where much of this work was carried out.

\section{Noncommutative correspondence categories} \label{NCC}

The general philosophy in this paradigm is that an object can be studied by its category of representations. This category of representations could be endowed with an abelian, triangulated, differential graded (DG) or some symmetric monoidal structure. It appears that dealing with abelian or triangulated categories is deficient from the homotopy theoretic point of view. It is better to work with the entire $\RHom$ complex of morphisms, which retains cochain level information rather than simply the zeroth cohomology group as in triangulated categories. Hence, from our perspective the appropriate category structure is that of a category enriched over cochain complexes, i.e., a DG category. 

It is often convenient to localize along certain morphisms for various purposes. For (co)homological constructions in geometry localizing along quasi-equivalences seems quite natural. This leads us to Keller's construction of the derived category of a DG category, which is essentially the category of modules (or representations) of the DG category up to homotopy. This derived category is itself is a triangulated category and for most practical purposes it suffices to work with this triangulated category of modules over the DG category or some suitable triangulated subcategory thereof. We recall some basic facts about DG categories and noncommutative geometry below. There is some freedom in choosing the way one would like to represent the known constructions in geometry in this setting. Let us reiterate that the noncommutative correspondence categories introduced below should be viewed simply as topological models.

\subsection{The category of small DG categories $\DGcat$}
The basic references for the background material, that we require, about the category of small DG categories are \cite{KelDG} and \cite{TabThesis}. In this setting the {\it noncommutative spaces} are viewed as small {\it DG categories}, i.e., categories enriched over the symmetric monoidal category of cochain complexes of $k$-linear spaces. Let $\DGcat$ stand for the category of all small DG categories. The morphisms in this category are {\it DG functors}, i.e., (enriched) functors inducing morphisms of $\Hom$-complexes. Henceforth, unless otherwise stated, all DG categories will be small. We provide one generic example of a class of DG categories which will be useful for later purposes.

\begin{ex} \label{DGex}
Given any $k$-linear category $\mathcal{M}$ it is possible to construct a DG category $\cC_{dg}(\mathcal{M})$ with cochain complexes $(M^\bullet,d_M)$ over $\mathcal{M}$ as objects and setting $\Hom(M^\bullet,N^\bullet) = \oplus_n \Hom(M^\bullet ,N^\bullet)_n$, where $\Hom(M^\bullet,N^\bullet)_n$ denotes the component of morphisms of degree $n$, i.e., $f_n:M^\bullet\map N^\bullet[n]$ and whose differential is the graded commutator $d_M\circ f_n - (-1)^nf_n\circ d_N$.  It is easily seen that the zeroth cocycle category $\Z^0(\cC_{dg}(\cM))$ reduces to the category of cochain complexes over $\cM$ and the zeroth cohomology category $\H^0(\cC_{dg}(\cM))$ produces the homotopy category of complexes over $\cM$. Thus, given any DG category $\cC$ the category $\H^0(\cC)$ is called the {\it homotopy} category of $\cC$.  If the objects of $\cC_{dg}(\cM)$ are taken to be chain (instead of cochain) complexes over $\cM$ one needs to set the $n$-th graded component of the morphism $\Hom(M_\bullet,N_\bullet)_n = \Hom(M_\bullet,N_\bullet[-n])$.
\end{ex}

 Now we recall the notion of the derived category of a DG category as in \cite{Kel}. Let $\cD$ be a small DG category. A right DG $\cD$-module is by definition a DG functor $M:\cD^{\op}\to
\cC_{dg}(k)$, where $\cC_{dg}(k)$ denotes the DG category of cochain complexes of
$k$-linear spaces. We denote the DG category of right DG modules over $\cD$ by $\textup{D}_{dg}(\cD)$. It generalizes the notion of a right module over an associative unital $k$-algebra $A$. Indeed, viewing $A$ as a category with one object $\ast$ such that $\End(\ast)=A$, a functor from the oposite category to $k$-linear spaces is just a $k$-linear space $M$ (the image of $\ast$) with a $k$-algebra homomorphism $A^{\op}\map\End_k (M)$ making $M$ a right $A$-module. Every object $X$ of $\cD$ defines
canonically what is called a {\it free} or {\it representable} right $\cD$-module $X^\wedge:= \Hom_{\cD}(-,X)$. A morphism of DG modules $f:L\to M$ is by definition a morphism (natural transformation) of DG functors such that $fX:LX\to MX$ is a morphism of complexes for all $X\in\textup{Obj}(\cD)$. We call such an $f$ a
quasi-isomorphism if $fX$ is a quasi-isomorphism for all $X$, i.e., $fX$
induces isomorphism on cohomologies. The derived category $\textup{D}(\cD)$ of $\cD$ is defined to be the localization of the category $\textup{D}_{dg}(\cD)$ with respect to the class of
quasi-isomorphisms. The category $\textup{D}(\cD)$ attains a triangulated structure with the translation induced by the shift of complexes and triangles coming
from short exact sequence of complexes. The Yoneda functor $X\mapsto X^\wedge$ induces an embedding $\H^0(\cD)\to \textup{D}(\cD)$. 

\begin{defn} \label{pretrDG}
The triangulated subcategory of $\textup{D}(\cD)$ generated by the free DG
$\cD$-modules $X^\wedge$ under translations in both directions, extensions and
passage to direct factors is called the {\bf perfect} derived category and
denoted by $\per(\cD)$. Its objects are called perfect modules. A DG category $\cD$ is said to be pretriangulated if the above-mentioned Yoneda functor induces an equivalence $\H^0(\cD)\to \per(\cD)$.
\end{defn}

\begin{rem} \label{repDG}
The homotopy category of a pretriangulated category has a triangulated category structure which is idempotent complete. There is a canonical DG version, denoted by $\per_{dg}(\cD)$, whose homotopy category is $\per(\cD)$. The construction is analogous to that of Example \ref{DGex}. Considering an associative unital algebra $A$ as a DG category one finds that $\per(A)$ is equivalent to the homotopy category of bounded complexes of finitely generated projective modules over $A$. 
\end{rem}

\subsection{The Morita model structure on $\DGcat$} \label{MoritaModel} A DG functor $F:\cC\to \cD$ is called a Morita morphism if it induces an exact equivalence $F^*:\textup{D}(\cD)\to \textup{D}(\cC)$. There are many equivalent formulations of a Morita morphism. The one that we have chosen is perhaps the most direct generalization to the derived setting of (algebraic) Morita morphism as an equivalence of module categories. Thanks to Tabuada \cite{TabThesis} we know that $\DGcat$ has a {\it cofibrantly generated Quillen model category} structure, where the weak equivalences are the Morita morphisms and the fibrant objects are pretriangulated DG categories. The category of noncommutative spaces $\NCS$ is defined to be the localization of $\DGcat$ with respect to the Morita morphisms, i.e., the Morita homotopy category of $\DGcat$. One should bear in mind that a Morita morphism is, in general, weaker than what is called a quasi-equivalence, which generalizes the concept of a quasi-isomorphism. Given any DG category $\A$ one constructs $\per_{dg}(\A)$ as its pretriangulated replacement or fibrant replacement.  Using the fibrant replacement functor it is possible to view $\NCS$ as a full subcategory of $\DGcat$ localized along quasi-equivalences consisting of pretriangulated DG categories. 

The category $\NCS$ has an inner Hom functor which we denote by $\rep_{dg}(-,?)$ \cite{ToeDG}. For the benefit of the reader we recall briefly its construction. Given any DG category $\cA$  one can construct a $\cC_{dg}(k)$-enriched model category structure on the DG category of right DG $\cA$-modules $\textup{D}_{dg}(\cA)$, whose homotopy category turns out to be equivalent to the derived category of $\cA$ \cite{ToeDG}. Let $\Int(\cA)$ denote the category of cofibrant-fibrant objects of this model category, which may be regarded as a $\cC_{dg}(k)$-enrichment of the derived category of $\cA$. If $\cC$ and $\cD$ are DG categories then their inner DG category $\rep_{dg}(\cC,\cD)$ is by definition $\Int(\textup{D}_{dg}(\cD^{\op}\otimes\cC))$. 

A DG functor $\phi:\cC\map\cD$ naturally gives rise to a $\cD^{\op}\otimes\cC$-module $M_\phi$, i.e., $M_\phi(-\otimes ?)=\Hom_\cD(\phi(?),-)$. This is one advantage of working in the DG setting, i.e., every morphism (not necessarily isomorphisms) in $\NCS$ becomes a generalized DG bimodule morphism or a noncommutative correspondence. In the geometric triangulated setting, e.g., when the triangulated category is of the form $\textup{D}^b(\Coh(X))$ for some smooth and proper variety $X$, such a result is true only for exact equivalences \cite{Orl1}.

\subsection{A convenient localization of $\DGcat$} Waldhausen's $\K$-theory construction $\bK$ produces a functor $\DGcat\functor\HoSpt$, where $\HoSpt$ is the (triangulated) homotopy category of spectra. More precisely, given any DG category $\cA$ one constructs a {\it Waldhausen category} structure on the category of perfect right $\cA$-modules with cofibrations as module morphisms which admit a (graded) retraction and weak equivalences as quasi-isomorphisms. Then one applies Waldhausen's machinery \cite{Waldhausen} to this category. The homotopy groups of this spectrum are by definition the Waldausen $\K$-theory groups of the DG category $\cA$. We may perform a localization $\mathcal{L}_{\bK}(\DGcat)$ with respect to the class of morphisms inverted by $\bK$, i.e., morphisms $f$ such that $(\bK(f))$ is a homotopy equivalence of spectra. 

The category $\DGcat$ is {\it combinatorial}, i.e., it is cofibrantly generated model category and its underlying category is locally presentable (see, e.g., \cite{AdaRos} for the generalities). Indeed, it is cofibrantly generated by construction and the underlying category is locally presentable as follows: if $\mathcal{V}$ is a locally presentable symmetric monoidal category then the category of small $\mathcal{V}$-enriched categories is also locally presentable \cite{KelLac} and the category of cochain complexes over $k$ is clearly locally presentable. Now any combinatorial model category is Quillen equivalent to a left proper and simplicial model category \cite{DugLoc}. Therefore we may define $\mathcal{L}_{\bK}(\DGcat)$ as a Bousfield homological localization \cite{BouLoc,BouLocSpa} and avoid set-theoretic problems. This drastic localization has the effect that all DG categories which are `indistinguishable at the level of $\K$-theory spectra up to homotopy' become isomorphic in the homotopy category $\mathtt{Ho}\mathcal{L}_{\bK}(\DGcat)$. The category $\mathtt{Ho}\mathcal{L}_{\bK}(\DGcat)$ enjoys the property that the $\K$-theory functors $\K_i:\DGcat\functor\Ab$ factor through it, where $\Ab$ is the category of abelian groups. 

We call a category {\it semiadditive} if it has a zero object, it has finite products and coproducts, the canonical map from the coproduct to the product (which exists thanks to the zero object) is an isomorphism and it is enriched over discrete commutative monoids.

\begin{lem}
The category  $\mathtt{Ho}\mathcal{L}_{\bK}(\DGcat)$ is semiadditive.
\end{lem}

\begin{proof}
It was shown in \cite{TabThesis} that the homotopy category of $\DGcat$ is pointed, i.e, it has a zero object and it has finite products and coproducts. The canonical map from a finite coproduct to the product of DG categories induces an isomorphism between their Waldhausen $\K$-theory spectra, since $\bK$ is an additive invariant \cite{DugShi,Tab3}. Therefore, this map is invertible in $\DGcorr$. The addition (commutative monoid operation) of morphisms $f,g:\cC\map\cD$ is defined by the following composition of arrows

\beqn
\cC\overset{\Delta}{\map}\cC\times\cC\overset{(f,g)}{\map}\cD\times\cD\cong\cD\coprod\cD\overset{\iota}{\map}\cD,
\eeqn where $\Delta$ is the diagonal map and $\iota$ is the fold map, induced by the universal property of the coproduct applied to two copies of the map $\id_A:A\map A$, i.e., it is the dual of the diagonal map. 

\end{proof}

Now we may apply the monoidal group completion functor to the morphism sets (enriched over monoids) of $\mathtt{Ho}\mathcal{L}_{\bK}(\DGcat)$ to obtain certain categories enriched over abelian groups. Since the monoids here are discrete a na\"ive group completion suffices. Products, coproducts and the zero object remain unaffected. Therefore, by construction we end up with an additive category. We define this additive category as our {\it noncommutative correspondence category} in this framework and denote it by $\DGcorr$. There is a canonical functor $\mathtt{Ho}\mathcal{L}_{\bK}(\DGcat)\functor\DGcorr$ which is identity on objects and sends each morphism monoid to its group completion via the canonical map. Since $\Hmo_0$ is the universal additive invariant \cite{Tab3} the functor $\bK$ on $\DGcat$ factors through $\Hmo_0$ and there is a commutative diagram (with additive functors)

\beqn
\xymatrix{
\Hmo_0
\ar[rr]^{\bK}
\ar@{-->}[dr]
&& \HoSpt \\
& \NCC_{\mathtt{dg}}^{\bK}
\ar[ur]_{\bK}}
\eeqn

\begin{rem}
An enriched (over $\HoSpt$) version of $\DGcorr$ can be obtained by performing the topological group completion given by $\Omega\textup{B}(-)$, which is the classifying space functor $\textup{B}$ followed by the loop functor $\Omega$.
\end{rem}

\subsection{Noncommutative $C^*$-correspondence category $\KKcat$}
The category of commutative separable $C^*$-algebras corresponds to that of metrizable topological spaces. Kasparov developed $\KK$-theory by unifying $\K$-theory and $\K$-homology into a bivariant theory and obtained interesting positive instances of the Baum--Connes conjecture \cite{KasKK1,KasKK2}. A remarkable feature of this theory is the existence of an associative Kasparov product on the $\KK$-groups. Higson proposed a categorical point of view of $\KK$-theory making use of the Kasparov product to define compositions \cite{Hig1}. The category of  $C^*$-algebras with morphisms enriched over Kasparov's bivariant $\KK_0$-groups plays the role of the category of noncommutative correspondences in the realm of noncommutative geometry {\it \`a la Connes} (see, for instance, \cite{ConSka,ConConMar1}). We denote this category by $\KKcat$. Morphisms in the $\KK_0$-groups can be expressed as homotopy classes of even Kasparov bimodules. Somewhat miraculously in the end one finds that all the analysis disappears and the morphisms are purely determined by topological data. Let us reiterate that it is not quite clear how geometric an abstract  $C^*$-algebra is. Connes provided a convenient framework of {\it spectral triples} to incorporate geometric structures into the picture \cite{ConDiff}. A promising candidate for the {\it category of spectral triples} has been put forward in \cite{Bram}. The category $\KKcat$ may be  regarded as a convenient model for the operator algebraic noncommutative correspondence category, where most of the well-known geometric examples fit in nicely. There is a canonical functor $\iota:\Csep\functor\KKcat$, where $\Csep$ is the category of  $C^*$-algebras with $*$-homomorphisms. The functor $\iota$ is identity on objects and makes the target of a $*$-homomorphism a bimodule in the obvious manner. Let $\cpt$ denote the algebra of compact operators on a separable Hilbert space. We set $A_\cpt:=A\otimes\cpt$.

\begin{rem} \label{Morita}
There is a counterpart of $\NCS$ in the world of  $C^*$-algebras, which we denote by $\CSp$. For the details we refer the readers to \cite{Mey1}, where it was called the category of {\it correspondences} in the operator algebraic setting. We regard this category as a category of noncommutative spaces where stably isomorphic algebras are identified. For separable $C^*$-algebras being stably isomorphic is equivalent to being Morita--Rieffel equivalent \cite{BGR}. The objects of $\CSp$ are  $C^*$-algebras and a morphism $A\map B$ is an isomorphism class of a right Hilbert $B_\cpt$-module $\mathcal{E}$ with a nondegenerate $*$-homomorphism $f:A_\cpt \map \cpt(\mathcal{E})$. There is a canonical functor $\Csep\functor\CSp$ which is the universal $C^*$-stable functor on $\Csep$ (Proposition 39 {\it loc. cit.}).
\end{rem}

In what follows we shall use $\KK$ (resp. $\K$) and $\KK_0$ (resp. $\K_0$) interchangeably.

\subsection{The passage from $\KKcat$ to $\DGcorr$}

For any  $C^*$-algebra $A$ the mapping $A\map A\otimes\cpt$ sending $a\longmapsto a\otimes \pi$, where $\pi$ is any rank one projection, is called the {\it corner embedding}. This map is clearly nonunital. A functor from $\Csep$ is called {\it $C^*$-stable} if the image of the corner embedding under the functor is an isomorphism. The definition of an {\it exact sequence} in $\Csep$ is simply a diagram isomorphic to $0\map I\map A\map A/I\map 0$, where $I$ is a closed two-sided ideal in $A$. Such a diagram is also known as an {\it extension diagram}. Since we are working with $C^*$-algebras, such extensions are {\it pure}, i.e., an inductive limit of $k$-module split extensions (see Theorem A.4. of \cite{Wod}). It is further called {\it split exact} if it admits a splitting $*$-homomorphism $s: A/I\map A$ (up to an isomorphism). A functor from $\Csep$ to a Quillen {\it exact} category is called {\it split exact} if it sends a split exact sequence of $C^*$-algebras to a distinguished short exact sequence in the target exact category. An abelian (resp. additive) category admits a natural exact structure, where the distinguished exact sequences are the natural short exact sequences (resp. direct sum diagrams). Higson proved that Kasparov's bivariant K-theory is the universal $C^*$-stable and split exact functor from $\Csep$ to an exact category, i.e., given any exact category $\cC$ and a $C^*$-stable and split exact functor $F:\Csep\map\cC$, there is a unique functor $\tilde{F}:\KKcat\functor \cC$ such that $\tilde{F}\circ\iota = F$ \cite{Hig1,Hig2}. Such a functor is automatically homotopy invariant \cite{Hig2}.

In this section we construct a covariant functor $\DGR:\KKcat\functor\DGcorr$. Our strategy would be to show that the functor $\DGR$ is a $C^*$-stable and a split exact functor on $\Csep$ so that we can apply Higson's Theorem to deduce that it factors through the category $\KKcat$.

The category $\CSp$ is the category of $C^*$-algebras with some generalized morphisms, in which Morita--Rieffel equivalent $C^*$-algebras become isomorphic (see Remark \ref{Morita}). A $C^*$-algebra $A$ is called {\it stable} if $A\cong A\otimes\cpt$, e.g., the algebra of compact operators $\cpt$ is itself stable. In $\CSp$ any $C^*$-algebra $A$ is isomorphic to a stable $C^*$-algebra functorially, {\it viz.,} its own stabilization $A_\cpt$. 

Given a split exact diagram in $\Csep$
\beqn
\xymatrix { 0\ar[r]&A\ar[r]^i & B \ar[r]^j& C\ar[r]\ar@/_1pc/[l]_s& 0}
\eeqn there exists a morphism $t:B\map A$ in $\KKcat$, which makes it a direct sum diagram. Since $\KK$-theory is morally the space of morphisms between $\K$-theories, we would like our DG category to be a categorical incarnation of $\K$-theory, even for nonunital algebras. 

Let us briefly recall a construction of Quillen \cite{QuiNonunitalK0}, which turns out to be useful to this end. Given any (possibly nonunital) $k$-algebra $A$, with unitization $\tilde{A}$, we consider the category $\Pr(A)$ whose objects are complexes $U$ of right $\tilde{A}$-modules, which are homotopy equivalent to bounded complexes of finitely generated projective modules over $\tilde{A}$, such that $U/AU$ is acyclic. We enrich the category $\Pr(A)$ over cochain complexes as explained in Example \ref{DGex} to make it a $k$-linear DG category. So the zeroth cocycle category $\Z^0(\Pr(A))$ forms a subcategory of the category of perfect complexes over $\tilde{A}$. Observe that $\Z^0(\Pr(A))$ is also a Waldhausen category with the weak equivalences (resp. cofibrations) pulled back from the associated Waldhausen category structure on the category of perfect right $\Pr(A)$-modules, which are precisely the homotopy equivalences (resp. monomorphisms with a graded splitting). The Grothendieck group of $\Z^0(\Pr(A))$ can be identified with the free abelian group generated by the homotopy classes of its objects and relations coming from short exact sequences of complexes, which are split in each degree, i.e., Waldhausen's $\K_0$-group. 

In general there is an exact functor from $\Z^0(\Pr(A))$ to the Waldhausen category of all perfect complexes over $\tilde{A}$ whose $\K$-theory spectrum is canonically homotopy equivalent to Quillen's algebraic $\K$-theory spectrum (obtained, for instance, by $Q$-construction). This induces a map of spectra $\bK(\Z^0(\Pr(A)))\map\bK^{\textup{alg}}(A)$. Observe that excision holds for algebraic $\K$-theory of $C^*$-algebras \cite{SusWod2} and so it makes sense to talk about the algebraic $\K$-theory spectrum of a nonunital $C^*$-algebra. As a result there is a canonical map at the level of Grothendieck groups  $\K_0(\Pr(A))\map\K_0(A):=\K_0(\tilde{A})/\K_0(k)$, where $\K_0(\tilde{A})$ (resp. $\K_0(k)$) can be identified with the Grothendieck group of stable isomorphism classes of finitely generated projective modules over $\tilde{A}$ (resp. $k)$. This map turns out to be an isomorphism when $A$ is a $C^*$-algebra (Proposition 6.3 in \cite{QuiNonunitalK0}). In order to understand the structure of the Grothendieck group one may work with a simpler subcategory of $\Pr(A)$. Consider the full subcategory of $\Pr(A)$ consisting objects $f:P\map Q$, where $P,Q$ are finitely generated projective modules over $\tilde{A}$, such that the induced map $\bar{f}:P/AP\map Q/AQ$ is an isomorphism. The canonical inclusion of this subcategory inside $\Pr(A)$ induces an isomorphism on their Grothendieck groups, which admits a simple presentation given in Theorem 8.4. {\it ibid.}

Any $*$-homomorphism $g:A\map B$ between possibly nonunital $C^*$-algebras extends uniquely to a unital map (preserving the adjoined unit) $\tilde{g}:\tilde{A}\map\tilde{B}$ between their unitizations. Then one can consider ${B}$ as an $\tilde{A}$-$\tilde{B}$-bimodule (left structure is given by the map $\tilde{g}$ and using the fact that $B$ is a two-sided ideal in $\tilde{B}$), which gives rise to a functor $g_*:= -\otimes_{\tilde{A}} B:\Pr(A)\map\Pr(B)$ that induces a well-defined map on $\K$-theory ({\it c.f.,} Theorem 3.1. {\it ibid.}). Here we take the algebraic tensor product. Therefore, our construction $A\mapsto \Pr(A)$ is functorial with respect to $*$-homomorphisms, i.e., $\Pr:\Csep\map\DGcat$ is a functor.

\begin{lem} \label{KStable}
The functor $\Pr:\Csep\functor\NCS$ is $C^*$-stable.
\end{lem}

\begin{proof}
It is known that if two separable (more generally $\sigma$-unital) $C^*$-algebras are stably isomorphic then they are Morita--Rieffel equivalent \cite{BGR}. Since $A$ and $A_\cpt$ are stably isomorphic there are, by definition, Morita--Rieffel equivalence bimodules ${}_A X_{A_\cpt}$ and ${}_{A_\cpt} X'_A$, satisfying certain conditions, see e.g., \cite{IndRepRie}. Then $X$ can be made into a left (resp. right) unitary module over the unitization $\tilde{A}$ (resp. $\tilde{A}_\cpt$) of $A$ (resp. $A_\cpt$) by setting $(a,\lambda)x = ax +\lambda x$ for $(a,\lambda)\in\tilde{A}$. Similarly one makes $X'$ into a left $\tilde{A}_\cpt$ and right $\tilde{A}$-module. One can check that $\left(\begin{smallmatrix} \tilde{A} & X \\ X' & \tilde{A}_\cpt\end{smallmatrix}\right)$ defines a Morita context. Now by Corollary 3.2. of \cite{QuiNonunitalK0} one deduces that $X\otimes_A -$ and $X'\otimes_{A_\cpt} -$ induce inverse equivalences between $\Pr(A)$ and $\Pr(A_\cpt)$ in $\NCS$. 
\end{proof}

As a consequence we obtain a functorial construction between noncommutative spaces.

\begin{cor}
There is an induced functor $\Pr:\CSp\functor\NCS$.
\end{cor}

\begin{rem} \label{excision}
In fact, Theorem 4.2. of \cite{QuiNonunitalK0} asserts that up to a Morita context $\Pr(\pounds)$ is independent of the embedding of $\pounds$ in a unital $C^*$-algebra as a closed two-sided ideal, which ensures that the functor $\bK$ satisfies excision, i.e., whenever $0\map A\map B\map C\map 0$ is an exact sequence of $C^*$-algebras $\bK(\Pr(A))\map\bK(\Pr(B))\map\bK(\Pr(C))$ is a homotopy fibration.
\end{rem}

Note that the category $\NCS$ does not involve any $\K$-theoretic localization. Since the $C^*$-stability property of $\Pr$ is actually achieved in $\NCS$ it is independent of the $\K$-theoretic localization. The localization will be needed now to prove its split exactness. 

\begin{lem} \label{splitExact}
The functor $\Pr:\Csep\functor\DGcorr$ is split exact.
 \end{lem}
 
\begin{proof}
For any split exact sequence

\beqn
\xymatrix { 0\ar[r]&A\ar[r]^i & B \ar[r]^j& C\ar[r]\ar@/_1pc/[l]_s& 0},
\eeqn applying $\DGR$ we obtain the diagram

\beqn
\xymatrix { \Pr(A)\ar[r]^{I=i_!} & \Pr(B) \ar[r]^{J=j_!}& \Pr(C)\ar@/_2pc/[l]_{S=s_!}},
\eeqn where $JI=0$ and $JS= \id_{\Pr(C)}$. One can construct $\ker(J)$ in the Karoubian closure of $\DGcorr$, because $SJ:\Pr(B)\map\Pr(B)$ is a projection, and obtain the following diagram

\begin{equation} \label{diagram}
\xymatrix{
&{\Pr(A)}
\ar@{-->}[d]^{\kappa}
\ar[dr]^{I} \\
0
\ar[r]
&\ker(J)
\ar[r]^{\mathcal{I}}
&{\Pr(B)}
\ar[r]^{J}
&{\Pr(C)}
\ar@/_2pc/[l]_{S}
\ar[r]
& 0,
}
\end{equation} where the existence of $\kappa$ follows from the universal properties of $\ker(J)$. Our aim is to show that $\kappa$ is an isomorphism in $\DGcorr$.

\noindent
We may apply Waldhausen's $\K$-theory spectrum functor $\bK$ to the above diagram \eqref{diagram} and by Remark \ref{excision} it follows that $\bK(\kappa)$ is an isomorphism in the homotopy category of spectra, whence $\kappa$ is an isomorphism in $\DGcorr$.

\end{proof}

Now for some technical benefits we modify our functor $\Pr$ slightly. For any $C^*$-algebra $\pounds$ we define $\DGR(\pounds)=\per_{dg}(\Pr(\pounds))$ (see Definition \ref{pretrDG} and the remark thereafter). 

\begin{rem}
For any unital $C^*$-algebra $\pounds$ the category $\DGR(\pounds)$ is our DG model (up to a derived Morita equivalence) for the bounded derived category of finitely generated projective right $\pounds$-modules.
\end{rem}

Observe that by definition the image of $\DGR$ is a $k$-linear pretriangulated DG category. The canonical Yoneda map $\theta_\pounds: \pounds\map\per_{dg}(\pounds)$ is an isomorphism in $\NCS$. One advantage of pretriangulated DG categories is that one can construct cones of morphisms functorially in them. Therefore, our final functor $\DGR$ is a composition of two functors $\per_{dg}\circ\Pr$. Since $\theta_\pounds:\pounds\map\per_{dg}(\pounds)$ is an isomorphism, a map $f:\pounds\map \pounds'$ in $\Csep$ induces an {\it unnatural} map $f_!:\per_{dg}(\Pr(\pounds))\map\per_{dg}(\Pr(\pounds'))$ (in the same direction) by setting $f_! = \theta_{\pounds'}f_*\theta_\pounds^{-1}$. Therefore, $\DGR(\pounds)$ is a covariant functor $\Csep\functor\NCS\functor\DGcorr$.

\begin{rem}
The application of $\per_{dg}$, i.e., taking a fibrant replacement also creates some flexibility to bring in more analysis into the picture. For instance, instead of taking DG functors with values in chain complexes over $k$ one could take functors with values in the category of chain complexes over the quasiabelian category of topological vector spaces \cite{Schn,Pros}. Conceivably one could still prove a result similar to Theorem \ref{MainThm} below, which we leave for the readers to figure out.
\end{rem}

\begin{lem} \label{KStable2}
The functor $\DGR:\Csep\functor\DGcorr$ is $C^*$-stable and split exact.
\end{lem}

\begin{proof}
The assertions follow from Lemma \ref{KStable} and Lemma \ref{splitExact} since $\per_{dg}(\pounds)$ is simply a fibrant replacement of $\pounds$ in the Morita model category $\DGcat$.
\end{proof}

\begin{lem} \label{Htpy}
The functor $\DGR:\Csep\functor\DGcorr$ is homotopy invariant.
\end{lem}

\begin{proof}
This is an immediate consequence of Theorem 3.2.2. of \cite{Hig2} which says that any $C^*$-stable and split exact functor on $\Csep$ is automatically homotopy invariant.
\end{proof}

Now we prove a Proposition which shows that our functor $\DGR$ encodes topological $\K$-theory. The proof is modelled along the lines of {\it ibid.}.

\begin{prop} \label{KSpec}
Let $A$ be any $C^*$-algebra. Then $\bK(\DGR(A))$ is homotopy equivalent to the connective cover $\bK^{\textup{top}}(A)\langle 0\rangle$ of the topological $\K$-theory spectrum.
\end{prop}

\begin{proof}
By Lemma \ref{KStable2} we may replace $A$ by $A_\cpt$ and use the fact that $\Pr(\pounds)\cong\DGR(\pounds)$ in $\NCS$. For the benefit of the reader we now recall a standard dimension shifting argument for stable $C^*$-algebras using the exact sequence $0\map\C_0((0,1))\otimes A_\cpt\map\C_0([0,1))\otimes A_\cpt\map A_\cpt\map 0$.
 
As discussed above there is a map of spectra $\bK(\Pr(A_\cpt))\map\bK^{\textup{alg}}(A_\cpt)$ which induces an isomorphism at the level of $\K_0$. Using Remark \ref{excision} we obtain the following map of exact sequences [set $\K^{\dg}_i(\pounds)=\pi_i(\bK(\Pr(\pounds)))$, $\C_0([0,1))\otimes A_\cpt) =\Cone A_\cpt$ and $\C_0((0,1))\otimes\pounds = \Sigma\pounds$]

\beqn
\xymatrix{
\K_1^{dg}(\Cone A_\cpt)
\ar[r]
\ar[d]
& \K_1^{\dg}(A_\cpt)
\ar[d]
\ar[r]
&\K_0^{\dg}(\Sigma A_\cpt)
\ar[d]
\ar[r]
& \K_0^{dg}(\Cone A_\cpt)
\ar[d]
\ar[r]
&\K_0^{\dg}(A_\cpt)
\ar[d] \\
 \K_1^{\textup{alg}}(\Cone A_\cpt) 
 \ar[r]
 &\K^{\textup{alg}}_1(A_\cpt)
 \ar[r]
 &\K^{\textup{alg}}_0(\Sigma A_\cpt)
 \ar[r]
 & \K_0^{\textup{alg}}(\Cone A_\cpt)
 \ar[r]
& \K^{\textup{alg}}_0(A_\cpt)\; .
 }
\eeqn We know that the three vertical arrows from the right are isomorphisms. Now we exploit the homotopy invariance of $\K_i^{dg}$ and $\K_i^{\textup{alg}}$ and the fact that $\Cone A_\cpt$ is contractible (see, e.g., Theorem 4.2.7. of \cite{Hig2}) to deduce that $\K_i^{dg}(\Cone A_\cpt)=\K_i^{\textup{alg}}(\Cone A_\cpt)=0$, whence the boundary maps $\K_1^{\textup{alg}}(A_\cpt)\map\K_0^{\textup{alg}}(\Sigma A_\cpt)$ and $\K_1^{\textup{dg}}(A_\cpt)\map\K_0^{\textup{dg}}(\Sigma A_\cpt)$ are isomorphisms. 
It follows immediately that $\K_1^{\dg}(A_\cpt)\map\K_1^{\textup{alg}}(A_\cpt)$ is an isomorphism. The isomorphisms $\K_i^\dg(A_\cpt)\map\K_i^{\textup{alg}}(A_\cpt)$ for $i\geqslant 2$ follow easily by induction.

Thanks to the Theorem of Suslin--Wodzicki \cite{SusWod2} we know that the algebraic $\K$-theory spectrum is (connectively) homotopy equivalent to the topological $\K$-theory spectrum of a stable $C^*$-algebra (see also, e.g., Theorem 1.4. of \cite{RosAlgKOperAlg}). Since $A_\cpt$ is stable, $\bK(\Pr(A_\cpt))$ is actually homotopy equivalent to the connective cover of the topological $\K$-theory spectrum of $A_\cpt$, which in turn is homotopy equivalent to $\bK^{\textup{top}}(A)\langle 0\rangle$.
\end{proof}

Now we state the main Theorem in this section.
 
 \begin{thm} \label{MainThm}
 The functor $\DGR$ factors through $\KKcat$; in other words, we have the following commutative diagram of functors: 
 \beqn
 \xymatrix{ 
\Csep
\ar[rr]^{\DGR}
\ar[dr]_{\iota}
&& \DGcorr\; .\\
& \KKcat
\ar@{-->}[ur]_{\DGR}}
\eeqn
 \end{thm}
 
 \begin{proof}
 We have already checked that the functor $\DGR$ is $C^*$-stable (Lemma \ref{KStable2}) and split exact (Lemma \ref{splitExact}). It remains to apply Higson's characterization of $\KK$ as the universal $C^*$-stable and split exact functor on $\Csep$ \cite{Hig1,Hig2}.
 \end{proof}

\begin{cor}
An isomorphism in $\KKcat$ implies an isomorphism in $\DGcorr$. In other words, $\KK$-equivalence implies (correspondence like) derived DG Morita equivalence up to a $\K$-theoretic identification, or Morita-$\K$ equivalence, for brevity. It is known that two unital rings with equivalent derived categories of modules have isomorphic (algebraic) $\K$-theories \cite{DugShi}. We have the following sequence of implications for  $C^*$-algebras

\beqn
\text{Morita--Rieffel equiv.}\Rightarrow \text{$\KK$-equiv.}\Rightarrow\text{Morita-$\K$ equiv.}\Rightarrow\text{isom. top. $\K$-theories}.
\eeqn
 \end{cor} 
 
 \begin{rem}
 Any $C^*$-stable and split exact functor on $\Csep$ satisfies Bott periodicity \cite{Hig2} (see also \cite{CunMeyRos}). Hence the functor $\DGR$ will also have this property.
\end{rem}
 
 It is useful to know that the functor $\DGR:\KKcat\functor\DGcorr$ exists by abstract reasoning. However, in order to make the situation a bit more transparent we make use of a rather algebraic formulation of $\KK$-theory due to Cuntz \cite{CunKK}. For any  $C^*$-algebra $A$ let $A\ast A$ denote the free product (which is the coproduct in $\Csep$) of two copies of $A$ and let $qA$ be the kernel of the fold map $A\ast A\map A$. Cuntz showed that $A$ (resp. $B$) is isomorphic to $qA$ (resp. $B\otimes\cpt$) in $\KKcat$ and $\KK(A,B)\cong[qA,B\otimes\cpt]$, i.e., homotopy classes of $*$-homomorphisms $qA\map B\otimes \cpt$. Roughly, the algebra $qA$ is expected to play the role of a cofibrant replacement of $A$ and $B\otimes \cpt$ that of a fibrant replacement of $B$ with respect to some model structure with $\KK_*$-equivalences as weak equivalences. This goal has been accomplished to some extent in \cite{JoaJoh}. The authors of {\it ibid.} work with the larger category of {\it $\nu$-complete l.m.c-$C^*$-algebras} and in their cofibrantly generated $\textup{KK}$-model structure, i.e., the model structure in which the weak equivalences are $\textup{KK}_*$-equivalences, all objects are fibrant and a minor modification of $qA$ acts as a cofibrant replacement.
 
 One benefit of this approach is that Kasparov's product can be viewed simply as a composition of $*$-homomorphisms, which is quite often easier to deal with. In order to define the abelian group structure one proceeds roughly as follows: for any $\phi, \psi\in [qA,B\otimes\cpt]$ one defines $\phi\oplus\psi : qA \map \mathbb{M}_2(B\otimes\cpt)$ as 
 $\left( \begin{smallmatrix}
 \phi & 0 \\
 0    &  \psi
 \end{smallmatrix}
 \right)$. Then one argues that $\mathbb{M}_2(B\otimes\cpt)$ is isomorphic to $B\otimes\cpt$ in $\KKcat$ and fixing such an isomorphism $\theta$ one sets $\phi + \psi := \theta(\phi\oplus\psi)$.

 \begin{prop} 
 The functor $\DGR:\KKcat\functor\DGcorr$ is additive.
 \end{prop}
 
 \begin{proof}
 We simply use the fact that any direct sum diagram in $\KKcat$ can be expressed as a split exact diagram involving only $*$-homomorphisms applying $q(-)$ and ${-}\otimes\cpt$ several times, both of which produce isomorphic objects in $\KKcat$, and  then apply the split exactness property of $\DGR$.

\end{proof}

 \begin{ex} \label{TorusEx}
Let $A= \C(E)$ be the $C^*$-algebra of continuous functions on a complex elliptic curve $E$. Topologically $E$ is isomorphic to a $2$-torus $\mathbb{T}^2$. It is known that $\K(A)$ is isomorphic to $\ZZ^2$. Using the fact that $A$ belongs to the Universal Coefficient Theorem class one computes $\KK(A,A)\simeq \textup{M}(2,\ZZ)$. 
 
The group of invertible elements can be identified with $\GL(2,\ZZ)$. Let $\textup{D}^b(E)$ be the bounded derived category of coherent sheaves on $E$. Since all autoequivalences of $\textup{D}^b(E)$ are geometric in nature (see, e.g., Theorem 3.2.2. of \cite{Orl1}), they definitely give rise to automorphisms of $\DGR(A)$; in other words, there is a group homomorphism $\Aut(D^b(E))\map\Aut(\DGR(A))$.  The automorphism group $\Aut(\textup{D}^b(E))$ can be described explicitly (see, e.g., Remark 5.13. (iv) \cite{BurKre}). It maps surjectively onto $\SL(2,\ZZ)$ with a non-canonical splitting defined by sending the generators of $\SL(2,\ZZ)$ to some specific Seidel--Thomas twist functors.

 
It is clear that the group $\Aut(\textup{D}^b(E))$ is bigger than $\GL(2,\ZZ)$ since it contains $\textup{Pic}^0(E)$ as a subgroup. The $\SL(2,\ZZ)$ part of $\Aut(\textup{D}^b(E))$ can be described by the Seidel--Thomas twist functors, which can also be seen as Fourier--Mukai transforms \cite{SeiTho} and it seems that $\KK$-equivalences can account for them. 
\end{ex} 

The automorphism groups of commutative $C^*$-algebras in $\KKcat$ are computable from their $\K$-theories using the Universal Coefficient Theorem \cite{RosSch}. As the example suggests, for any locally compact Hausdorff topological space $X$ the group $\Aut_{\KKcat}(\C_0(X))$ always maps to $\Aut_{\DGcorr}(\DGR(\C_0(X))$ and gives some idea about the automorphisms of the DG derived category of topological vector bundles of $X$. However, the automorphism group in $\DGcorr$ will typically be larger than the automorphism group of derived categories consisting of exact equivalences up to a natural transformation.

 \subsection{$\E$-theory and strong deformations up to homotopy} \label{deformation}
 The Connes--Higson $\E$-theory \cite{ConHig} is the universal $C^*$-stable, exact and homotopy invariant functor. Recall that a functor $F$ is {\it exact} if it sends an exact sequence of $\C^*$-algebras $0\map A\map B\map C\map 0$ to an exact sequence $F(A)\map F(B)\map F(C)$ (exact at $F(B)$). It is known that any exact and homotopy invariant functor is split-exact. Therefore, there is a canonical induced functor $\KKcat\functor\mathtt{E}$, where the category $\mathtt{E}$ consists of $C^*$-algebras with bivariant $\E_0$-groups as morphisms. This functor is fully faithful when restricted to nuclear $C^*$-algebras; in fact, by the Choi--Effros lifting Theorem $\KK_*(A,B)\cong\E_*(A,B)$ whenever $A$ is nuclear. Thus, restricted to nuclear $C^*$-algebras there are maps $\E_0(A,B)\cong\KK_0(A,B)\map \Hom_{\DGcorr}(\DGR(A),\DGR(B))$. 
 
 Let $B_\infty:= \C_b([1,\infty),B)/\C_0([1,\infty),B)$, where $\C_b$ denotes bounded continuous functions, be the {\it asymptotic algebra} of $B$. An {\it asymptotic morphism} between $C^*$-algebras $A$ and $B$ is a $*$-homomorphisms $\phi:A\map B_\infty$.
 
A {\it strong deformation} of $C^*$-algebras from $A$ to $B$ is a continuous field $A(t)$ of $C^*$-algebras over $[0,1]$ whose fibre at $0$, $A(0)\cong A$ and whose restriction to $(0,1]$ is the constant field with fibre $A(t)\cong B$ for all $t\in (0,1]$. 

Given any such strong deformation of $C^*$-algebras and an $a\in A$ one can choose a section $\alpha_a(t)$ of the continuous field such that $\alpha_a(0) =a$. Suppose one has chosen such a section $\alpha_a(t)$ for every $a\in A$. Then one associates an asymptotic morphism by setting $(\phi(a))(t) = \alpha_a(1/t), \; t\in[1,\infty)$. Let $\Sigma A:=\C_0((0,1),A)$ be the {\it suspension} of $A$. Then the Connes--Higson picture of $\E$-theory says that $\E_0(A,B) \cong [[\Sigma A\otimes\cpt,\Sigma B\otimes\cpt]]$, where $[[?,-]]$ denotes homotopy classes of asymptotic morphisms between $?$ and $-$. Let $\phi$ be an asymptotic morphism defined by a strong deformation from $A$ to $B$. Then we call the class of $\phi$ in $\E_0(A,B)$ a {\it strong deformation up to homotopy} from $A$ to $B$. If $A$ is nuclear, e.g., if $A$ is commutative, whenever the class of $\phi$ is invertible in $\E_0(A,B)$, we deduce that $\DGR(A)$ and $\DGR(B)$ are equivalent in $\DGcorr$.

\subsection{Homological $\TT$-dualities} \label{T-duality} A {\it sigma model} roughly studies maps $\Sigma\map X$, where $\Sigma$ is called the {\it worldsheet} (Riemann surface) and $X$ the {\it target spacetime} (typically a $10$-dimensional manifold in supersymmetric string theories). Mirror symmetry relates the sigma models of type $\mathtt{IIA}$ and $\mathtt{IIB}$ string theories with dual Calabi--Yau target spacetimes. In open string theories, i.e., when $\Sigma$ has boundaries, the boundaries are constrained to live in some special submanifolds of the spacetime $X$. Such a submanifold also comes equipped with a special {\it Chan-Paton} vector bundle and together they define a topological $\K$-theory class of $X$. This class is also known as a {\it D-brane charge} in physics literature. The homological mirror symmetry conjecture of Kontsevich predicts an equivalence of triangulated categories of $\mathtt{IIA}$-branes (Fukaya category) on a Calabi--Yau target manifold $X$ and $\mathtt{IIB}$-branes (derived category of coherent sheaves) on its dual $\hat{X}$. This equivalence would induce an isomorphism between their Grothendieck groups and it was argued that the Grothedieck group of the category of $A$-branes on $X$ should be isomorphic to $\K_1^{\textup{top}}(\hat{X})$ \cite{Wit}. Strominger--Yau--Zaslow argued that sometimes when $X$ and $\hat{X}$ are mirror dual Calabi--Yau $3$-folds one should be able to find a generically $\TT^3$-fibration over a common base $Z$

\begin{equation} \label{fibration}
\xymatrix{
X
\ar@{-->}[rr]^{\text{$\TT$-duality}}
\ar[dr]
&& \hat{X} \, ,
\ar[dl] \\
& Z
}
\end{equation}

such that mirror symmetry is obtained by applying $\TT$-duality fibrewise \cite{SYZ}. Since $\TT$-duality is applied an odd number of times it interchanges types ($\mathtt{IIA}\leftrightarrow \mathtt{IIB}$). Sometimes using Poincar{\' e} duality type arguments it is possible to identify topological $\K$-theory with $\K$-homology. Kasparov's $\KK$-theory naturally subsumes $\K$-theory and $\K$-homology and it was shown in \cite{BMRS2} that certain topological $\TT$-duality transformations (even including more parameters like $\H$-fluxes, which we did not discuss here) can naturally be seen as $\KK_1$-classes between suitably defined continuous trace $C^*$-algebras capturing the geometry of the above diagram \ref{fibration}. If $\TT$-duality is applied an even number of times it will preserve types. As argued above, whilst an odd number of $\TT$-duality transformations induces a shift in topological $\K$-theory, an even number preserves it due to Bott periodicity and naturally corresponds to a $\KK_0$-class. Therefore, a topological $C^*$-correspondences or a $\KK_0$-class is an abstract generalization of an even number of $\TT$-duality transformations (or $\TT^{2n}$-dualities), viewed as an equivalence of $\mathtt{IIB}$-branes on the same target manifold. Since $\KK$-theory is Bott periodic one can also use the identification $\KK_1(\pounds,\pounds ')\cong \KK_0(\pounds,\C_0((0,1))\otimes\pounds ')$.
 
 \section{Simplicial sets and pro $C^*$-algebras} \label{Homotopy}
 
 In this section we construct a pro $C^*$-algebra from a simplicial set and show that the construction is functorial with respect to {\it proper} maps between simplicial sets and pro $C^*$-algebras. We also show that this construction respects homotopy of proper maps after stabilizing the category of pro $C^*$-algebras with respect to finite matrices.

\subsection{Generalities on simplicial sets and pro $C^*$-algebras} \label{ProCAlg}
The standard reference for simplicial aspects of topology is \cite{GoeJar}. Let $\Delta$ be the cosimplicial category, i.e., the category whose objects are finite ordinals $[n]:=\{0,1,\cdots , n\}$ and whose morphisms are monotonic nondecreasing maps. Its morphisms admit a unique decomposition in terms of {\it coface} and {\it codegeneracy} maps, which satisfy certain well-known relations. Let $\Sets$ be the category of all sets. By a {\it simplicial set} we mean a functor $\Sigma:\Delta^{\op}\functor \Sets$ and a morphism of simplicial sets is a natural transformation between these functors. We denote by $\SSets$ the category of simplicial sets. The elements of $\Sigma[n]$ are called the $n$-simplices (or $n$-dimensional simplices) and the images of the coface and codegeneracy maps in $\Delta$ are called the {\it face} (denoted by $d_i$) and {\it degeneracy} maps (denoted by $s_j$). An $n$-simplex $\sigma$ is called {\it degenerate} if it is of the form $\sigma=s_i(\tau)$ for some $(n-1)$-simplex $\tau$. Degenerate simplices are needed to ensure that maps of graded sets exist, even if the target simplicial set has no nondegenerate simplex in a particular dimension. Simplicial sets provide a combinatorial description of topology. The singular simplices functor and the geometric realization functor are adjoint functors between the category of compactly generated and Hausdorff topological spaces and that of simplicial sets, which induce inverse equivalences between their homotopy categories with respect to their natural model category structures. 

A {\it pro $C^*$-algebra} is a complete Hausdorff topological $*$-algebra over $\CC$ whose topology is determined by its continuous $C^*$-seminorms, i.e., a net $\{a_\lambda\}$ converges to $0$ if and only if $p(a_\lambda)$ tends to $0$ for every $C^*$-seminorm $p$ on it. For any pro $C^*$-algebra $B$ let us denote the closure of $A\subset B$ by $\overline{A}$. It is known that every pro $C^*$-algebra has an approximate identity (Corollary 3.12 \cite{NCP1}). A morphism of pro $C^*$-algebras is a $*$-homomorphism and we do not require them to be continuous. We call such a morphism $f:A\map B$ {\it proper} if $\overline{f(A)B} = B$. These maps are also known as {\it nondegenerate} maps in the literature. Such maps were considered in the context of $C^*$-algebras in \cite{EilLorPed} and they correspond to proper maps between locally compact spaces under Gelfand--Naimark duality. Let us denote the category of pro $C^*$-algebras with proper $*$-homomorphisms by $\PAlg$. Of course, the category of $C^*$-algebras with proper $*$-homomorphisms is a full subcategory of $\PAlg$. Typical examples of commutative pro $C^*$-algebras are of the form $\C(X)$, i.e., complex valued continuous functions on a compactly generated Hausdorff space $X$ with the topology of uniform convergence on compact subsets. Note that such an algebra is always unital. Given any pro $C^*$-algebra $A$ and a $C^*$-seminorm $p$ on it, define the closed ideal $\ker(p)=\{a\in A\, |\, p(a)=0\}$. Then $A/\ker(p)$ is a $C^*$-algebra and the $C^*$-seminorms naturally form a directed family, such that $A\cong \varprojlim_p \; A/\ker(p)$, where this limit is taken in the category of topological $*$-algebras with continuous $*$-homomorphisms and not in $\Csep$, which also has all small limits and colimits. Given any pro $C^*$-algebra $A$ one can define a new pro $C^*$-algebra $\C([0,1],A)$, where every $C^*$-seminorm $p$ on $A$ defines a $C^*$-seminorm $p'$ on $\C([0,1],A)$ by $p'(f) = \textup{sup}_{x\in [0,1]} p(f(x))$. This enables us to define the notion of homotopy of maps between two pro $C^*$-algebras. Given two morphisms $f_1,f_2:A\rightarrow B$ in $\PAlg$ we say that $f_1$ is homotopic to $f_2$ (written as $f_1\sim f_2$) if there is a commutative diagram in $\PAlg$
\beqn
 \xymatrix{ 
 & B \\
A 
\ar[ur]^{f_1}
\ar[r]^{h}
\ar[dr]_{f_2} 
& \C([0,1],B)
\ar[u]_{\ev_0}
\ar[d]^{\ev_1} \\ 
& B } \\
\eeqn and we call $h$ the homotopy map.

Many other well-known constructions available at the level of $C^*$-algebras can also be performed in the category of pro $C^*$-algebras. More details on pro $C^*$-algebras can be found in \cite{NCP1,NCP2}. One of the main results of \cite{NCP1} is that the category of commutative unital pro $C^*$-algebras is equivalent to the category of {\it quasitopological spaces}, which contains the category of compactly generated spaces as a full subcategory. 

\begin{rem}
In general, a $*$-homomorphism between pro $C^*$-algebras need not be automatically continuous. However, if the domain is a $\sigma$-$C^*$-algebra, i.e., its topology is determined by a countable family of $C^*$-seminorms, then automatic continuity holds (Theorem 5.2 of \cite{NCP1}). The category of commutative unital $\sigma$-$C^*$-algebras with unital $*$-homomorphisms is equivalent to that of {\it countably compactly generated} spaces, i.e., spaces which appear as a countable direct limit of compact spaces, with continuous maps.
\end{rem}

\subsection{Simplicial sets and posets}

Let us denote the set of nondegenerate simplices  of $\Sigma$ by $\ND(\Sigma)$. Also set $\ND_{\geqslant 1}(\Sigma):=\{\sigma\in\ND(\Sigma)\,|\, \textup{dim}(\sigma)\geqslant 1\}$. Note that $\ND(\Sigma)=\Sigma[0]\cup\ND_{\geqslant 1}(\Sigma)$. Any degenerate simplex $\sigma$ can be written as $s_1\cdots s_n(\tau)$, where $\tau$ is a nondegenerate simplex which is uniquely determined by $\sigma$. Then $\sigma$ is said to be a {\it degeneracy} of $\tau$. Given any simplicial set $\Sigma$ we define a map (of sets) 
\beqn
\chi: \Sigma&\map& \ND(\Sigma)\\
\sigma &=&\begin{cases} \sigma \text{ if $\sigma$ is nodegenerate,}\\
                     \tau \text{ if $\sigma=s_{i_1}\cdots s_{i_n}(\tau)$ and $\tau$ nondegenerate.}
                     \end{cases}
                     \eeqn

The set $\ND(\Sigma)$ is actually a poset once we set $\sigma_1\leqslant \sigma_2$ if $\sigma_1$ is an iterated face of $\sigma_2$. Any poset can be viewed as a category and thus we may construct the nerve of $\ND(\Sigma)$, which we denote by $B(\Sigma)$. The assigment $\Sigma\mapsto\ND(\Sigma)$ is functorial as any map $f:\Sigma_1\map\Sigma_2$ in $\SSets$ induces a map (functor) of posets $f_*:\ND(\Sigma_1)\map\ND(\Sigma_2)$ sending $\sigma \mapsto \chi(f(\sigma))$. Note that $\textup{dim}(f_*(\sigma))\leqslant \textup{dim}(\sigma)$ and $f_*(\Sigma_1[0])\subset\Sigma_2[0]$.

\subsection{The construction of a pro $C^*$-algebra from a simplicial set}
The idea of the construction presented here is inspired by the one of Cuntz \cite{CunSimp} in the context of locally finite simplicial complexes.

A quiver is an oriented graph with a set of vertices $V$ and oriented edges $E$. There are two maps $s,t:E\map V$ such that if $x=v_1\map v_2$ is an edge in $E$ then $s(x)=v_1$ (source map) and $t(x)=v_2$ (target map). Any  poset can be viewed as a quiver. Our construction is a slight modification of the path algebra of a quiver. Let us denote by $\overline{\ND(\sd(\Sigma))}$ the quiver obtained by adjoining for every edge $x$ in the poset $\ND(\Sigma)$ an edge $x^*$ in the opposite direction, i.e., $s(x)=t(x^*), t(x)=s(x^*)$. The quiver $\overline{\ND(\Sigma)}$ will have the same set of vertices as the poset $\ND(\Sigma)$, but unlike $\ND(\Sigma)$ will also have oriented cycles. The vertices of $\overline{\ND(\Sigma)}$ correspond to nodegenerate simplices of $\Sigma$. This quiver may have infinitely many vertices and edges. An oriented path $\gamma$ of length $n$ in a quiver is just a sequence of edges $\gamma=x_1\cdots x_n$ such that $t(x_i) = s(x_{i+1})$ for all $i=1,\cdots n-1$. The source and the target maps can be extended to paths by setting $s(\gamma)=s(x_1)$ and $t(\gamma)=t(x_n)$. Let $\Lambda$ be the partially ordered set consisting of finite subsets of vertices of $\overline{\ND(\Sigma)}$ (ordered by inclusion of finite subsets). 

An element $a$ of a pro $C^*$-algebra $A$ is called {\it positive} and denoted $a\geq 0$ if $a=u^*u$ for some element $u\in A$. Now one can construct a pro $C^*$-algebra $\Pro\C^*(\Sigma)$ as the universal pro $C^*$-algebra with positive generators $v$, for every vertex $v$ of $\ND(\Sigma)$ and some other generators $x$, for every oriented edge $x$ in $\ND(\Sigma)$, and $x^*$ its formal adjoint denoting the edge in the opposite direction in $\overline{\ND(\Sigma)}$, satisfying the relations:

\begin{enumerate}

\item $xy= \begin{cases} \text{concatenation of $x$ and $y$ if $t(x)=s(y)$,}\\
                                              0 \text{ otherwise,} 
                                              \end{cases}$ \label{Cond1}

\item $xx^* = s(x)$, $x^*x=t(x)$ \label{Cond5} \\

\item $vx = \begin{cases} x \text{ if $s(x)=v$,}\\
                            0 \text{ otherwise,}
                            \end{cases}$ \label{Cond3}
                            
\item $xv = \begin{cases} x \text{ if $t(x)=v$,}\\
                     0 \text{ otherwise,}
                     \end{cases}$ \label{Cond4}

\item $\textup{lim}_{\lambda\in\Lambda}\,\sum_{v\in\lambda} vw = w$ for all vertices $w\in\ND(\Sigma)$, \text{ (convergence in any $C^*$-seminorm)}. \label{Cond2}

\end{enumerate} 

N. C. Phillips introduced certain types of relations in \cite{NCP2}, which he called {\it weakly admissible} (Definition 1.3.4. {\it ibid.}) and it was shown that any set of generators with such relations admits a universal pro $C^*$-algebra (Proposition 1.3.6. {\it ibid.}). Our relations are readily seen to verify all the conditions of weak admissibility. Weakly admissible relations are only expected to be preserved under finite products of representations in $C^*$-algebras, as opposed to arbitrary products. 

\begin{rem}
Depending on one's taste one might also want to regard $\Pro\C^*(\Sigma)$ as a modified graph pro $C^*$-algebra. Note that the monomials in the generators (or the oriented paths containing only edges in the poset $\ND(\Sigma)$) correspond to the nondegenerate simplices of the nerve of the poset $\ND(\Sigma)$, which is closely related to the subdivision of $\Sigma$. Each oriented path records some information about the manner in which the target simplex is connected to other simplices. However the vertices are not viewed as trivial loops and they do not act as a family of orthogonal projections. The finite sums of the generators corresponding to the vertices are made to act like an approximate identity, in other words, a noncommutative partition of unity. Therefore, if $\Sigma$ has countably many nondegenerate simplices then $\Pro\C^*(\Sigma)$ is $\sigma$-unital. 
\end{rem}

The above remark suggests that the algebra $\Pro\C^*(\Sigma)$ essentially models the nerve of the poset $\ND(\Sigma)$, which we denote by $B(\Sigma)$. In fact, it does a little more. It models the nerve of the groupoid obtained by formally inverting all morphisms in $\ND(\Sigma)$. With some foresight, we require a {\it regularity} property on our simplicial sets. This property will play no role in this article but for some of the applications that we have in mind it is good to impose this condition. Therefore, given any simplicial set $\Sigma$ we first subdivide it to form a regular simplicial set $\sd(\Sigma)$ and then apply the construction $\Pro\C^*$ to it. Subdivision of a simplicial set is a functorial construction and more details about it can be found in, e.g., \cite{GoeJar}. Moreover, it is known that $\Sigma$ and $B(\Sigma)$ need not have the same homotopy type. The simple example of a simplicial circle $\Sigma$ with only one nondegenerate $0$-simplex and one nondegenerate $1$-simplex, whose boundaries are identified with the unique $0$-simplex produces a $B(\Sigma)$ which is contractible. Therefore, sending $\Sigma\mapsto B(\Sigma)$ is not a good operation from the point of view of homotopy theory. However, it is known that if $\Sigma$ is a {\it regular} simplicial set, then $B(\Sigma)$ has the right homotopy type. 

\noindent
We define our functor $\cRep$ as one whose map on objects is the composition $\Pro\C^*\circ\sd$. Its behaviour on morphisms will be defined below.

\begin{lem} \label{normBound}
If $\sd(\Sigma)$ has finitely many nondegenerate simplices $v_1,\cdots , v_n$ then $p(v_i)=1$ for any $C^*$-seminorm $p$ on $\cRep(\Sigma)$.
\end{lem}

\begin{proof}
By definition $v_i\geq 0$ for all vertices $v_i\in\ND(\Sigma)$, whence $u v_i u\geq 0$ for any $u\in\ND(\sd(\Sigma))$ since $x\geq 0 \Rightarrow y^*xy\geq 0$ and $u^*=u$. Thus we obtain $\sum_i v_i v_j v_i = v_i^2$ which implies that $ v_i v_j v_i \leqslant v_i^2$ for all $j=1,\cdots ,n$. For the term corresponding to $j=i$ we get $v_i^2 - v_i^3\geq 0$ and $v_i$ commutes with it, whence $v_i(v_i^2 - v_i^3)\geq 0$. Thus, $v_i^4 \leqslant v_i^3\leqslant v_i^2$ implying $p(v_i^4)=p(v_i)^4 \leqslant p(v_i^2)=p(v_i)^2$. This shows that $p(v_i)\leqslant 1$ and the relations \eqref{Cond3} and \eqref{Cond4} imply that $p(v_i)\geqslant 1$ for any $v_i\in\ND(\Sigma)$.
\end{proof}

\begin{rem}
If $\Sigma$ is {\it finite}, i.e., $\Sigma$ has only finitely many nondegenerate simplices, which implies that the geometric realization of $\Sigma$ is compact and Hausdorff, then $\sd(\Sigma)$ also has finitely many nondegenerate simplices and $\cRep(\Sigma)$ is unital with the sum of all the vertices of $\ND(\sd(\Sigma))$ acting as the identity. 

Thus, from the above Lemma, we conclude that for finite simplicial sets (or compact topological spaces) the construction actually yields a genuine unital $C^*$-algebra by putting a bound on the norms of all the generators. Indeed, the norms of the generators corresponding to the edges are bound by those of the vertices, which is evident from the relation \eqref{Cond5}. 
\end{rem}

For any pro $C^*$-algebra $A$ let $b(A)$ denote the set of {\it bounded elements} in $A$, i.e., $b(A)=\{a\in A\,|\, \textup{sup}_p \,p(a) <\infty, \text{ $p$ $C^*$-seminorm on $A$}\}$. It is shown in \cite{NCP2} that $b(A)$ equipped with this supremum norm is a $C^*$-algebra. If $A=\C(X)$, i.e., the algebra of continuous functions on some topological space $X$, then $b(A)=\C_b(X)$, i.e., the algebra of bounded continuous functions on $X$. The construction $A\mapsto b(A)$ is functorial with respect to $*$-homomorphisms (not necessarily continuous). Therefore, one can compose our functor $\cRep$ with the functor $A\mapsto b(A)$ to obtain a genuine $C^*$-algebra, but, this construction will lose a lot of information if the simplicial set is not finite. 

\begin{ex} 
Let $\Delta^n:=\Hom_\Delta(-,[n])$. Then $\Delta^0$ is a simplicial set with only one degenerate simplex in each dimension (except in dimension $0$, which has only one element) and one finds $\cRep(\Delta^0)\cong\CC$. 
\end{ex}

\begin{ex} \label{ProC}
Now $\Delta^1$ is a simplicial set with two $0$-simplices $\{a,b\}$ and one nondegenerate $1$-simplex, which we call $c$. The poset $\ND(\sd(\Delta^1))$ looks like

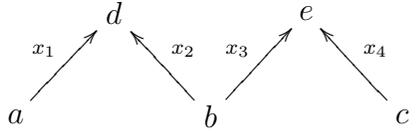
\beqn
\xymatrix{
&d
&&e\\
a
\ar[ur]^{x_1}
&& b
\ar[ul]_{x_2}
\ar[ur]^{x_3}
&&c
\ar[ul]_{x_4}
},
\eeqn where $a,b,c$ are the $0$-simplices and $d,e$ the nondegenerate $1$-simplices. Then the graph $\overline{\ND(\sd(\Delta^1))}$ is 

 \beqn
\xymatrix{
&d
\ar@/^/[dl]^{x_1^*}
\ar@/_/[dr]_{x_2^*}
&&e
\ar@/^/[dl]^{x_3^*}
\ar@/_/[dr]_{x_4^*}\\
a
\ar@/^/[ur]^{x_1}
&& b
\ar@/_/[ul]_{x_2}
\ar@/^/[ur]^{x_3}
&&c
\ar@/_/[ul]_{x_4}
}.
\eeqn

The univeral $C^*$-algebra $\cRep(\Delta^1)$ is described by the obvious relations coming from multiplications of paths in this quiver, e.g., $x_1 x_2 =0$, $bx_2=x_2$, $x_2^*b = x_2^*$, $(a+b+c+d+e) =1$ and so on. On the other hand, the $C^*$-algebra $\C([0,1])$ is the universal $C^*$-algebra with a presentation $\mathcal{R}:=\{1,y\,|\, y\geqslant 0, 1\geqslant 0, \|y\|\leqslant 1,\|1-y^2\|\leqslant 1, 1y=y1=y, 1^2=1\}$. Sending $y\mapsto\id$ and $1\mapsto 1$ defines a representation of the generators in $\C([0,1])$ and hence a map $\C^*(\mathcal{R})\map\C([0,1])$. The inverse map is obtained by applying continuous functional calculus to the element $y\in\C^*(\mathcal{R})$. 
\end{ex}

\subsection{Functoriality with respect to proper maps} \label{ProperMaps}

The construction $\cRep$ is functorial only with respect to a genuine subset of morphisms of simplicial sets. We propose the following definition of a {\it proper} map between simplicial sets.

\begin{defn} \label{PropMap}
A map $f:\Sigma_1\map\Sigma_2$ of simplicial sets is {\it proper} if the induced map $\ND(f):\ND(\Sigma_1)\map\ND(\Sigma_2)$ between posets is proper, i.e., preimage of any element in the target poset is finite. 
\end{defn}

For instance, suppose $f:\Sigma\map\Delta^0$ is a map of simplicial sets such that $\Sigma$ has infinitely many nondegenerate simplices. The poset of $\Delta^0$ consists of just its $0$-simplex and its preimage is all of $\ND(\Sigma)$ under the induced map on posets. Therefore, such a map is not proper as it should be. Topologically it is mapping a noncompact space to a point. Since $\ND:\SSets\functor\mathtt{Posets}$ is a functor, it can be checked that the composition of proper maps is again proper so that simplicial sets with proper maps form a category. We denote the category of simplicial sets with proper maps by $\SSpro$. The subdivision of a finite simplicial set is once again finite. If $f:\Sigma_1\map\Sigma_2$ is a proper map of simplicial sets then $\sd(f)$ is also proper as the preimage of any element in $\ND(\Sigma_2)$ is at most the subdivision of the finite simplicial set generated by the finitely many preimages of $f$ in $\ND(\Sigma_1)$, i.e., $\sd:\SSpro\functor\SSpro$ is a functor. Given any proper map of simplicial sets $f:\Sigma_1\map\Sigma_2$ one defines the induced map $\tilde{f}:\cRep(\Sigma_2)\map\cRep(\Sigma_1)$ by its value on the generators $v$ [vertex in $\ND(\sd(\Sigma_2))$] and $x,x^*$ [edges in $\overline{\ND(\sd(\Sigma_2))}$]

$\tilde{f}(v) = \begin{cases} \sum w, \text{  $w\in f^{-1}(v)$}, \\
                                                                                    0 \text{ if $f^{-1}(v)=\emptyset$,}
                          \end{cases}$

$\tilde{f}(x) = \begin{cases} \sum y, \text{  $y$ edge such that $s(y)\in f^{-1}(s(x))$, $t(y)\in f^{-1}(t(x))$}, \\
                                                                                    0 \text{ if no such $y$ exists,}
                          \end{cases}$ 
                          
$\tilde{f}(x^*)=\tilde{f}(x)^*$.
                          
By the properness of $f$ the sum on the right hand side in the first two cases is always finite. It is easy to check that relations \eqref{Cond1}, \eqref{Cond5}, \eqref{Cond3} and \eqref{Cond4} are satisfied and relation \eqref{Cond2} is verified by the following Proposition. For any simplicial set $\Sigma$ let us write $V(\Sigma)$ for the set of vertices of $\ND(\sd(\Sigma))$. 

\begin{prop}
Whenever $\textup{lim}_\lambda\sum_{v\in\lambda} vw= w$ in $\C^*(\Sigma_2)$, $\textup{lim}_\lambda\sum_{v\in\lambda}\tilde{f}(v)\tilde{f}(w)=\tilde{f}(w)$ in $\C^*(\Sigma_1)$, where $\lambda$ runs through finite subsets of $V(\Sigma_2)$. 
\end{prop}

\begin{proof}
Let us set $a_{\Omega} = \sum_{\omega\in\Omega} \omega$ for any finite $\Omega\subset V(\Sigma)$. For any $b\in\ND(\Sigma_1)$ and given any $\epsilon>0$, there is a finite $\Omega\subset V(\Sigma_1)$ such that $p( a_{\Omega'} b - b) < \epsilon$ for all $\Omega'\supset\Omega$ and $C^*$-seminorm $p$ on $\C^*(\Sigma_1)$. Then $f(\Omega)$ is a finite subset of $\Sigma_2$ and $\tilde{f}(a_{f(\Omega)}) = \sum_{f(v)\in f(\Omega)} v$. Since $f^{-1}(f(\Omega))\supset\Omega$ we obtain $p(\tilde{f}(a_{{\Omega''}})b -b) =p(\sum_{w\in{\Omega''}}\tilde{f}(w)b -b) < \epsilon$ for every ${\Omega''}\supset f(\Omega)$. Now we may replace $b$ by $\tilde{f}(w)$, $w\in V(\Sigma_2)$ to deduce $\textup{lim} \sum_{v\in\lambda} \tilde{f}(v)\tilde{f}(w) =\tilde{f}(w)$.

It is clear that if $f$ is a map between finite simplicial sets then the induced map $\tilde{f}$ is a unital map of unital $C^*$-algebras.
\end{proof}

\begin{lem}
Then map $\tilde{f}$ is a proper map of pro $C^*$-algebras.
\end{lem}

\begin{proof}
The proof of the above Proposition actually shows that $\{a_\lambda=\sum_{v\in\lambda} \tilde{f}(v)\}$, as $\lambda$ runs through finite subset of $\Sigma_2$, acts as an approximate identity in $\cRep(\Sigma_1)$. Since each $a_\lambda b\in f(\cRep(\Sigma_2))\cRep(\Sigma_1)$, one finds that for any $b\in \cRep(\Sigma_1)$, $b=\lim_\lambda a_\lambda b$ lies in $\overline{f(\cRep(\Sigma_2))\cRep(\Sigma_1)}$. 
\end{proof}

\begin{cor}
The construction $\cRep$ defines a functor $(\SSpro)^{\op}\map\PAlg$.
\end{cor}

\subsection{$\mathbb{M}_n$-stabilization and proper homotopy diagrams}
Let $A$ be any algebra and $\mathbb{M}_n(A)$ the algebra of $n\times n$ matrices over $A$. There is a corner embedding map $A\map \mathbb{M}_n(A)$ sending $a\mapsto \textup{M}_{ij}$, with $\textup{M}_{11}=a$ and $\textup{M}_{ij}=0$ otherwise. We would like to formally invert such maps. Clearly such maps are not proper, but their {\it formal inverses} are! So we enlarge the morphisms of $\PAlg$ by adjoining such maps and all compositions. Then we formally invert the corner embeddings and denote the $\mathbb{M}_n$-stabilized category by $\SPAlg$.

\begin{rem}
The collection of all corner embeddings is a genuine class of morphisms and if one inverts them one might potentially run into set-theoretic problems. Restricted to class of $\sigma$-$C^*$-algebras, it is clear that the resulting category is a subcategory of the model category consisting of {\it $\nu$-sequentially complete l.m.c. $C^*$-algebras} as constructed in \cite{JoaJoh} and the localization can be carried out inside this model category. For more general pro $C^*$-algebras this localization is potentially set-theoretically problematic and makes sense only in some enlarged universe.
\end{rem}

The product of two simplicial sets $\Sigma_1\times\Sigma_2$ is defined to be $(\Sigma_1\times\Sigma_2)[n] =\Sigma_1[n]\times\Sigma_2[n]$ with coordinatewise face and degeneracy maps. Note that this is not the categorical product in $\SSpro$. In fact, products  do not exist in $\SSpro$ in general. It is possible that $(\sigma,\sigma')$ be nondegenerate in $\Sigma_1\times\Sigma_2$ although both $\sigma$ and $\sigma'$ are degenerate in $\Sigma_1$ and $\Sigma_2$ respectively, as $\sigma$ and $\sigma'$ can be degeneracies by dissimilar sequences of degeneracy maps.

If $f_1, f_2:\Sigma_1\map\Sigma_2$ are two maps between simplicial sets then a simplicial homotopy between $f_1$ and $f_2$ is a commutative diagram 

\begin{equation} \label{SimpHomotopy}
\hspace{10mm} \xymatrix{ 
 {\Delta^0\times\Sigma_1=\Sigma_1} 
\ar[dr]^{f_1} 
\ar[d]_{d^0\times 1} \\
\Delta^1\times\Sigma_1
\ar[r]^{\gamma}
& \Sigma_2\; ,\\
{\Delta^0\times\Sigma_1=\Sigma_1}
\ar[u]^{d^1\times 1}
\ar[ur]_{f_2}\\}. 
\end{equation} where $d^0,d^1$ are maps induced by the coface operators and $\gamma$ is called the homotopy operator. Simplicial homotopy of maps is not an equivalence relation in general. It is so if the target simplicial set is fibrant. So one considers the equivalence relation generated by simplicial homotopy relation and by taking the equivalence classes of maps one obtains a na\"ive homotopy category of simplicial sets. If in the above diagram all the maps are proper we obtain a {\it proper homotopy diagram}. Note that the maps $d^i\times 1$, $i=1,2$ and $\id$ are always proper. We denote by $\PHoSS$ the category of simplicial sets with proper homotopy classes of proper maps between them, i.e., we identify two maps $f_1$ and $f_2$ whenever there is a proper homotopy diagram as above.

\begin{prop}
If $f_1,f_2:\Sigma_1\map\Sigma_2$ are properly homotopic then $\tilde{f_1}$ and $\tilde{f_2}$ are homotopic in $\SPAlg$.
\end{prop}

\begin{proof}
Let $\gamma$ be a homotopy between $f_1$ and $f_2$. Then applying the contravariant functor $\cRep$ to the simplicial homotopy diagram \eqref{SimpHomotopy} above we obtain

\beqn
 \xymatrix{ 
 & \cRep(\Sigma_1) \\
\cRep(\Sigma_2) 
\ar[ur]^{\tilde{f}_1}
\ar[r]^{\tilde{\gamma}}
\ar[dr]_{\tilde{f}_2} 
& \cRep({\Delta^1}\times\Sigma_1)\; .
\ar[u]_{\tilde{d^0}\times\tilde{1}}
\ar[d]^{\tilde{d^1}\times\tilde{1}} \\ 
& \cRep(\Sigma_1) }. \\
\eeqn Let us define a map 
\beqn
\eta:\cRep(\Delta^1\times\Sigma_1)&\map &\C([0,1],\mathbb{M}_2(\cRep(\Sigma_1))) \\
              {\delta} &\mapsto & \left[t\mapsto \omega_t\left(\begin{smallmatrix}(\tilde{d^0}\times\tilde{1}){(\delta)} & 0\\ 0 &  (\tilde{d^1}\times\tilde{1}){(\delta)}\end{smallmatrix}\right)\omega_t^{-1}\right],
              \eeqn where $\delta$ is any generator of $\cRep(\Delta^1\times\Sigma_1)$, $t\in[0,1]$ and $\omega_t = \left(\begin{smallmatrix}\cos(\frac{\pi t}{2}) & \sin(\frac{\pi t}{2})\\-\sin(\frac{\pi t}{2}) & \cos(\frac{\pi t}{2})\end{smallmatrix}\right)$ is the rotation homotopy matrix. 
                            
Upon identifying $\C([0,1],\mathbb{M}_2(\cRep(\Sigma_1)))\simeq \C([0,1],\cRep(\Sigma_1))$ via the formal {\it proper} inverse of the corner embedding and composing $\eta$ with it, we obtain a map $\cRep(\Sigma_2)\map \C([0,1],\cRep(\Sigma_1))$, which we continue to denote by $\eta$. It is readily seen that $\ev_{t=i}\circ\eta = \tilde{d^i}\times\tilde{1}$, $i=0,1$. Thus one obtains the required homotopy diagram (with homotopy operator $\eta\circ\tilde{\gamma}$)

\beqn
 \xymatrix{ 
 & \cRep(\Sigma_1) \\
\cRep(\Sigma_2) 
\ar[ur]^{\tilde{f}_1}
\ar[r]^{\tilde{\gamma}}
\ar[dr]_{\tilde{f}_2} 
& \cRep(\Delta_1\times\Sigma_1)
\ar[u]_{\tilde{d^0}\times\tilde{1}}
\ar[d]^{\tilde{d^0}\times\tilde{1}}
\ar[r]^{\eta}
& \C([0,1],\cRep(\Sigma_1))\; .
\ar[ul]_{\ev_{t=0}}
\ar[dl]^{\ev_{t=1}} \\ 
& \cRep(\Sigma_1) } \\
\eeqn 

\end{proof} As a consequence we obtain our main Theorem in this section. Let us denote the category of pro $C^*$-algebras with homotopy classes of maps in $\SPAlg$ by $\HoPA$.

\begin{thm}  \label{HomInv}
The functor $\cRep$ induces a functor between the categories $(\PHoSS)^{\op}$ and $\HoPA$.
\end{thm}

The essential image of this functor is possibly an interesting class of noncommutative pro $C^*$-algebras. Its objects admit simplicial descriptions. We expect the matrix stabilized category of pro $C^*$-algebras with proper maps to be a natural domain to develop a {\it noncommutative proper homotopy theory}.


\bibliographystyle{abbrv}
\bibliography{/Users/smahanta/Professional/math/MasterBib/bibliography}

\vspace{5mm}
\noindent
\address{Email: {\sf smahanta@ihes.fr}}

\end{document}